\documentclass[12pt, a4]{amsart}
\usepackage{amsgen,amsmath,amstext,amsbsy,amsopn,amsfonts,amssymb}
\usepackage{amsmath,amssymb,amsfonts,amsthm,enumerate}
\usepackage{amsthm}
\usepackage[utf8]{inputenc}
\usepackage{color}
\usepackage[pagebackref]{hyperref}
\usepackage[alphabetic]{amsrefs}
\usepackage{xcolor}
\usepackage{tikz-cd}

\usepackage{rotating}

\UseRawInputEncoding

\newtheorem{theorem}{Theorem}
\newtheorem{lemma}[theorem]{Lemma}

\newtheorem{proposition}[theorem]{Proposition}
\newtheorem{obs}[theorem]{Observation} \newtheorem{defi}[theorem]{Definition}

\newtheorem{exa}[theorem]{Example}

\newtheorem{rem}[theorem]{Remark}
\newenvironment{remark}{\begin{rem}\rm}{\end{rem}}
\newtheorem{rems}[theorem]{Remarks}

\newtheorem{ack}[theorem]{Acknowlegment}

\def\ZZ{{\mathbf Z}}
\def\CCC{{\mathbf C}}

\def\QQ{\mathbf Q}

\def\AA{{\mathbf A}}
\def\RR+{{\mathbf R}^*}

\def\Q_p{{\mathbf Q}_p}

\def\Proj{\rm Proj}

\def\eps{\varepsilon}
\def\Ga{\Gamma}

\def\la{\lambda}

\def\vfi{\varphi}

\def\Aut{{\rm Aut}}

\def\Ker{{\rm Ker}}

\newcommand{\supp}{\operatorname{supp}}

\newcommand{\Ind}{\operatorname{Ind}}

\newcommand{\Bohr}{\operatorname{Bohr}}
\newcommand{\Prof}{\operatorname{Prof}}

\newcommand{\Sym}{\operatorname{Sym}}

\def\tout{\qquad\text{for all}\quad}

\newtheorem{thmx}{Theorem}
 
\newtheorem{corx}[thmx]{Corollary}
\begin{document}

\title[ Bohr and profinite compactification of extensions]{On  Bohr compactifications and profinite completions of group extensions}

\address{Bachir Bekka \\ Univ Rennes \\ CNRS, IRMAR--UMR 6625\\
Campus Beaulieu\\ F-35042  Rennes Cedex\\
 France}
\email{bachir.bekka@univ-rennes1.fr}

\author{Bachir Bekka}

\thanks{The author acknowledges the support  by the ANR (French Agence Nationale de la Recherche)
through the project Labex Lebesgue (ANR-11-LABX-0020-01) .}
\begin{abstract}
Let $G= N\rtimes H$ be a  locally compact group which is a
semi-direct product of  a closed normal subgroup $N$ and a closed subgroup $H.$
The Bohr compactification $\Bohr(G)$  and the profinite completion   $\Prof(G)$ of $G$ 
are, respectively,   isomorphic to semi-direct products $Q_1 \rtimes \Bohr(H)$
and  $Q_2 \rtimes \Prof(H)$ for  appropriate quotients $Q_1$ of $\Bohr(N)$ and $Q_2$ 
of $\Prof(N).$ 
We give a precise description of $Q_1$ and $Q_2$ in terms of the action of $H$ on appropriate subsets
of the dual space of $N$. In the case where $N$ is abelian,  we have $\Bohr(G)\cong A \rtimes \Bohr(H)$
and  $\Prof(G)\cong B \rtimes \Prof(H),$ where  $A$
is the  group of unitary characters of $N$ with finite $H$-orbits   
 and $B$ the subgroup of $A$ of characters with finite image.
Necessary and sufficient conditions are deduced for $G$ to be maximally almost periodic
or residually finite.
We apply the results to the case where $G= \Lambda\wr H$ is a wreath product
of countable groups; we show in particular that $\Bohr(\Lambda\wr H)$ is isomorphic to  
$\Bohr(\Lambda^{\rm Ab}\wr H)$
and  $\Prof(\Lambda\wr H)$ is isomorphic to $\Prof(\Lambda^{\rm Ab} \wr H),$
where  $\Lambda^{\rm Ab}=\Lambda/ [\Lambda, \Lambda]$ is the abelianization of $\Lambda.$  
As examples, we compute    $\Bohr(G)$  and $\Prof(G)$ 
when  $G$ is a lamplighter group and when $G$ is the Heisenberg
group over a unital commutative ring.
\end{abstract}
\subjclass[2000]{22D10; 22D25; 22E50; 20G05}
\maketitle
\section{Introduction}
There are two distinguished compact groups associated to a general topological group $G.$  A    \textbf{Bohr compactification} 
(respectively, a \textbf{profinite completion}) of $G$ is  a pair  
consisting of a   compact  (respectively, profinite) group $K$ and a continuous homomorphism $\beta: G \to K $   with dense image  satisfying the following universal property: for every compact group (respectively, profinite group) $L$
and every continuous homomorphism $\alpha: G \to L$,
there exists a continuous homomorphism $\alpha':  K\to L$
such that the diagram 
\[
\begin{tikzcd}
& K\arrow[dashed]{d}{\alpha'} \\
G \arrow{ur}{\beta} \arrow{r}{\alpha} & L
\end{tikzcd}
\]
commutes.
Bohr compactifications and profinite completions   $(K, \beta) $ of $G$ are unique in the following sense:
if  $(K', \beta')$ is a pair consisting of a compact (respectively, profinite)  group $K'$
and a continuous homomorphism $\beta': G \to K'$ with dense image
satisfying the same universal property, then there exists an isomorphism $f: K \to K'$ of topological groups
such that $\beta' = f \circ \beta$.
Concerning  existence,  we give below (Proposition~\ref{Pro-General2}) models of  Bohr compactifications and profinite completions.
For more  on Bohr compactifications,  see  \cite[\S 16]{Dixm--C*}, \cite[4.C]{BH} or  \cite[Chap.VII]{Weil}; for more details on  profinite completions, see \cite{Ribes-Zal}.

We will often denote by $(\Bohr(G), \beta_G)$ and $(\Prof(G), \alpha_G)$
a Bohr compactification and a profinite completion of $G.$
In the sequel, for two topological groups $H$ and $L,$ we write $H\cong L$ if   $H$ and $L$ are topologically isomorphic.

The universal property of $\Bohr(G)$ gives rise to a continuous surjective homomorphism
$\alpha: \Bohr(G) \to \Prof(G)$ such that  $\alpha_G = \alpha \circ \beta_G$.
 It is easy to see (see \cite[Proposition 7]{Bekka1}) that the kernel of $\alpha$ is 
$\Bohr(G)_0$, the connected component of $\Bohr(G)$ and so
$$\Prof(G)\cong \Bohr(G)/ \Bohr(G)_0.$$

Every continuous homomorphism $G_1 \xrightarrow{f} G_2$ of topological groups
induces  continuous homomorphisms
$$\Bohr(G_1) \xrightarrow{\Bohr(f)} \Bohr(G_2) \qquad \text{and}\qquad\Prof(G_1) \xrightarrow{\Prof(f)} \Prof(G_2)$$ such that 
$\beta_{G_2} \circ f= \Bohr(f)\circ\beta_{G_1}$ and $\alpha_{G_2} \circ f= \Prof(f)\circ\alpha_{G_1}.$

Consider the category \textbf{TGrp} of  topological groups, with objects
the  topological groups and morphisms  the continuous homomorphisms between 
 topological groups.
 The Bohr compactification and the profinite completion   are covariant  functors 
 $$\Bohr: \textbf{TGrp} \to \textbf{CGrp} \qquad \text{and}\qquad \Prof:\textbf{TGrp} \to \textbf{PGrp}$$
 from  \textbf{TGrp} to the subcategory  \textbf{CGrp}
 of compact groups and the subcategory \textbf{PGrp} of profinite groups:
 
Assume that we are given an extension
\[
\begin{tikzcd} 
1\arrow{r}&N\arrow{r}{i}&G\arrow{r}{p}&G/N\arrow{r} &1. 
\end{tikzcd}
\leqno{(*)}
\]
of topological groups.
The functors $\Bohr$ and $\Prof$ are right exact and so the  diagrams
\[
\begin{tikzcd} 
\Bohr(N)\arrow{r}{\Bohr(i)}&\Bohr(G)\arrow{r}{\Bohr(p)}&\Bohr(G/N)\arrow{r} &1. 
\end{tikzcd}
\]
and 
\[
\begin{tikzcd} 
\Prof(N)\arrow{r}{\Prof(i)}&\Prof(G)\arrow{r}{\Prof(p)}&\Prof(G/N)\arrow{r} &1. 
\end{tikzcd}
\]
are exact; this  means that 
\begin{itemize}
\item $\Bohr(p)$ and $\Prof(p)$ are surjective and 
\item $\Ker (\Bohr(p))=\overline{\beta_G(N)}$   and $\Ker (\Prof(p))=\overline{\alpha_G(N)},$
\end{itemize}
where $\overline{A}$  denotes the closure of a subset $A$;
these facts are well-known and easy to prove (see, e.g., \cite[Lemma 2.2]{Hart-Kunen} and \cite[Proposition 3.2.5]{Ribes-Zal}; see also Proposition~\ref{Pro-Rightexact} below).
However, the functors  $\Bohr$ and $\Prof$ are not left  exact,
that is, $\Bohr(i): \Bohr(N) \to \Bohr(G)$ and $\Prof(i): \Prof(N) \to \Prof(G)$ are in general not injective (see e.g.
the examples given by Corollaries~\ref{Cor3} and \ref{Cor4} below).

For now on, we will deal only with \emph{locally compact} groups.
and with \emph{split} extensions. So, we will consider  
locally compact groups 
$G=N\rtimes H$ which are a semi-direct product of a normal closed subgroup $N$ and a closed subgroup $H$.
It is easy to see that $\Bohr(G)$, respectively $\Prof(G)$, is a  semi-direct product of 
$\overline{\beta_G(N)}$ with $\overline{\beta_G(H)}$, respectively of $\overline{\alpha_G(N)}$ with $\overline{\alpha_G(H)}$
(see \cite{Junghenn2}, \cite{Grunewald}). Our  results give a precise description 
of  the structure of these semi-direct products.

Denote by  $\widehat{N}_{ \rm fd}$  the  set of equivalence classes  (modulo unitary equivalence) of irreducible finite dimensional  unitary representations  of $N$. Every such representation $\sigma: N\to U(n)$ gives rise to  the
 unitary representation  $\Bohr({\sigma}):\Bohr(N)\to U(n) $   of $\Bohr(N)$;
 here (and elsewhere) we identify $\Bohr(U(n))$ with $U(n).$

  Observe that $H$ acts  on $ \widehat{N}_{\rm fd}$: for $\sigma\in \widehat{N}_{\rm fd}$ and $h\in H,$ 
the conjugate representation $\sigma^h \in \widehat{N}_{\rm fd}$
is defined by $\sigma^ h(n)=\sigma( h^{-1} n h)$ for all $n\in N.$

Define  $\widehat{N}_{ \rm fd}^{H-{\rm per}}$ as  the set of  
$\sigma\in \widehat{N}_{ \rm fd}$ with \emph{finite} $H$-orbit.

Observe that, due to the universal property of  $\Bohr(N)$, the group
  $H$ acts by automorphisms on  $\Bohr(N)$. However, this action does not extend
  in general to an action of $\Bohr(H)$ on $\Bohr(N).$
  
  Our first result show that $\Bohr(G)$ is a split extension of $\Bohr(H)$ by an appropriate quotient of $\Bohr(N).$
 
\begin{thmx}
\label{Theo1}
 Let $G=N\rtimes H$ be a semi-direct of locally compact groups.
Let $\vfi_N: \Bohr(N)\to \overline{\beta_G(N)}$ and $\vfi_H: \Bohr(H)\to \overline{\beta_G(H)}$ be the maps such that
$\vfi_N\circ \beta_N= \beta_G|_N$ and $\vfi_H\circ \beta_H= \beta_G|_H$
Set 
$$
C:=\bigcap_{\sigma \in \widehat{N}_{ \rm fd}^{H-{\rm per}}} \Ker (\Bohr(\sigma)).
$$
\begin{itemize}
\item[(i)] We have  $\Ker \vfi_N =C$ 
and so $\vfi_N$ induces a topological isomorphism $\overline{\vfi_N}: \Bohr(N)/C \to \overline{\beta_G(N)}.$
\item[(ii)] $\vfi_H: \Bohr(H)\to \overline{\beta_G(H)}$ is a topological isomorphism.
\item[(iii)] The action of $H$ by automorphisms on  $\Bohr(N)$ induces an action of 
$\Bohr(H)$ by automorphisms on  $\Bohr(N)/C$ and the maps $\overline{\vfi_N}$ and $\vfi_H$
give rise to an  isomorphism $$ \Bohr(G)\cong (\Bohr(N)/C) \rtimes \Bohr(H).$$
\end{itemize}
\end{thmx}

We turn to the description of $\Prof(G).$  Let $\widehat{N}_{ \rm finite} $ be the set    of irreducible 
 unitary representations $\sigma$ of $N$ with finite image $\sigma(N).$
  Observe that the action of $H$ on  $ \widehat{N}_{\rm fd}$ preserves $\widehat{N}_{ \rm finite}.$
  Let  $\widehat{N}_{ \rm finite}^{H-{\rm per}}$ be the subset
of $\widehat{N}_{ \rm finite}$ of representations with finite $H$-orbit.
Every  $\sigma\in \widehat{N}_{ \rm finite}$ gives rise to  the
 unitary representation  $\Prof({\sigma}) $   of $\Prof(N)$.

 A   result  completely similar to Theorem~\ref{Theo1} holds for  $\Prof(G)$.
 \begin{thmx}
\label{Theo2}
Let  $G=N\rtimes H$ be a semi-direct of locally compact groups.
Let $\psi_N: \Prof(N)\to \overline{\alpha_G(N)}$ and $\psi_H: \Prof(H)\to \overline{\alpha_G(H)}$ be the maps such that
$\psi_N\circ \alpha_N= \alpha_G|_N$ and $\psi_H\circ \alpha_H= \alpha_G|_H$
Set 
$$
D:=\bigcap_{\sigma \in \widehat{N}_{ \rm finite}^{H-{\rm per}}} \Ker (\Prof({\sigma})).
$$
\begin{itemize}
\item[(i)] We have  $\Ker \psi_N =D$ 
and so $\psi_N$ induces a topological isomorphism $\overline{\psi_N}: \Prof(N)/D \to \overline{\alpha_G(N)}.$
\item[(ii)] $\psi_H: \Prof(H)\to \overline{\alpha_G(H)}$ is a topological isomorphism.
\item[(iii)] The action of $H$ by automorphisms on  $\Prof(N)$ induces an action of 
$\Prof(H)$ by automorphisms on  $\Prof(N)/D$ and the maps $\overline{\psi_N}$ and $\psi_H$
give rise to an  isomorphism $$\Prof(G) \cong (\Prof(N)/D) \rtimes \Prof(H) .$$
\end{itemize}
\end{thmx}
When $N$ is a finitely generated (discrete) group,  we obtain the following well-known result (see \cite[Proposition 2.6]{Grunewald}).

 \begin{corx}
\label{Cor0}
Assume that  $N$ is  finitely generated. Then  $\Prof(G) \cong \Prof(N) \rtimes \Prof(H) .$
\end{corx}

In the case where $N$ is abelian, we can give a more  explicit description of the quotients $\Bohr(N)/C$ and $\Prof(N)/D$
appearing in Theorems~\ref{Theo1} and ~\ref{Theo2}.
Recall that, in this case, the dual group $\widehat{N}$ is the group  
of continuous homomorphisms from $N$ to the circle group   $\mathbf{S}^1$.
We will also consider the subgroup $\widehat{N}_{\rm fin}$ of $\chi\in \widehat{A}$
with finite image $\chi(N)$, that is, with values in the subgroup  of $m$-the roots of unity in $\mathbf{C}$
for some integer $m\geq 1.$
Observe also that $\widehat{N}^{H-{\rm per}}$ and $\widehat{N}_{ \rm finite}^{H-{\rm per}}$
are subgroups of $\widehat{N}$.
\begin{corx}
\label{Cor1}
Assume that  $N$ is an \emph{abelian} locally compact  group. 
Let  $\widehat{N}^{H-{\rm per}}$ and $\widehat{N}_{ \rm finite}^{H-{\rm per}}$
be equipped with the discrete topology.
Let $A$ and $B$ be their respective dual groups. Then 
$$\Bohr(G) \cong A \rtimes \Bohr(H)  \qquad\text{and} \qquad \Prof(G) \cong B \rtimes \Prof(H) .$$
\end{corx}

Recall that  $G$  is \textbf{maximally almost periodic}, or \textbf{MAP},
if  $\widehat{G}_{ \rm fd}$ separates its points (equivalently, if 
$\beta_G: G\to \Bohr(G)$ is injective); recall also that $G$ is \textbf{residually finite},
or \textbf{RF},
if  $\widehat{G}_{ \rm finite}$ separate its points (equivalently, if 
$\alpha_G: G\to \Prof(G)$ is injective).

\begin{corx}
\label{Cor2} 
Let $G=N\rtimes H$ be a semi-direct of locally compact groups.
\begin{itemize}
\item[(i)] $G$ is MAP if and only if $H$ is MAP and $\widehat{N}_{ \rm fd}^{H-{\rm per}}$
separates the points of $N.$
\item[(ii)] $G$ is RF  if and only if $H$ is RF and $\widehat{N}_{ \rm finite}^{H-{\rm per}}$
separates the points of $N.$
\end{itemize}
\end{corx}

We give an application of  our results to  the case where 
$G= \Lambda\wr H$ is the \textbf{wreath product} of the countable groups  $\Lambda$ and $H.$
Recall that $G$ is  the semidirect product $(\oplus_{h\in H}\Lambda)\rtimes H,$
where $H$ acts on $\oplus_{h\in H}\Lambda$ by shifting the indices.

We denote by $\Lambda^{\rm Ab}$ the abelianization $\Lambda/[\Lambda, \Lambda]$ of $\Lambda.$
Observe that $\Lambda^{\rm Ab}\wr H$ is a quotient of $\Lambda\wr H$

\begin{corx}
\label{Cor3}
Let $\Lambda$ and $H$ be countable groups.
\begin{itemize}
\item[(i)] When $H$ is finite, we have
\[\begin{aligned}
&\Bohr(\Lambda\wr H)\cong \left( \oplus_{h\in H}\Bohr(\Lambda)\right) \rtimes \Bohr(H) \ \text{and}\\
&\Prof(\Lambda\wr H)\cong \left( \oplus_{h\in H}\Prof(\Lambda)\right) \rtimes \Prof(H).
\end{aligned}
\]
\item[(ii)] When $H$ is infinite, 
the quotient map $\Lambda\wr H\to \Lambda^{\rm Ab}\wr H$ induces isomorphisms
$$
\Bohr(\Lambda\wr H)\cong \Bohr( \Lambda^{\rm Ab}\wr H) \ \text{and}\ \Prof(\Lambda\wr H)\cong \Prof( \Lambda^{\rm Ab}\wr H)
$$
In particular, if $\Lambda$ is perfect (that is, $\Lambda=[\Lambda, \Lambda]$), the quotient map $\Lambda\wr H\to  H$
induces  isomorphisms
$$
\Bohr(\Lambda\wr H)\cong \Bohr(H) \ \text{and}\ \Prof(\Lambda\wr H)\cong \Prof( H).
$$
\end{itemize}
\end{corx}

 Item (ii) of the following result is Theorem 3.2 in  \cite{Gruenberg}.
\begin{corx}
\label{Cor4}
Let $\Lambda$ and $H$ be countable groups.
Assume that $\Lambda$ has at least two elements.
\begin{itemize}
\item[(i)] $\Lambda\wr H$ is MAP  if and only if  $\Lambda$ is MAP
and $H$ is  RF and if either $H$ is finite or $\Lambda$ is abelian.
\item[(ii)] \textbf{\cite{Gruenberg}} $\Lambda\wr H$ is RF  if and only if $\Lambda$ and $H$ are  both RF and if
either $H$ is finite  or $\Lambda$ is abelian.
\end{itemize}
\end{corx} 

\begin{remark}
\label{Rem1}
\begin{itemize}
\item[(i)]The Bohr compactification of an abelian locally compact group $A$  is easy to describe: $\Bohr(A)$  can be identified with $\widehat{\Gamma},$ where $\Gamma= \widehat{A}$ viewed as discrete group;
in case $A$ is finitely generated, a more precise description of $\Bohr(A)$  is available (see \cite[Proposition 11]{Bekka1}).
\item[(ii)] Provided $\Bohr(H)$ and  $\Prof(H)$ are known, Corollary~\ref{Cor3} together with Corollary~\ref{Cor1} 
give, in view of (i), a  complete description of the Bohr compactification and the profinite completion of \emph{any}
wreath product $\Lambda\wr H$ in case $H$ is infinite.
\item[(iii)] Bohr compactifications of group and semigroup  extensions have been studied by several authors,
in a  more abstract and less explicit  setting (\cite{Dangello}, \cite{Junghenn1}, \cite{Junghenn2}, \cite{Junghenn3}, \cite{Landstad}, \cite{Milnes}); profinite completions of group extensions appear  at numerous places
in the literature (\cite{Grunewald},  \cite{Ribes-Zal}).

\end{itemize}
 \end{remark} 
 
  This paper is organized as follows. Section~\ref{S:General} 
 contains some general facts about  Bohr compactifications and profinite completions
 as well as some reminders  on projective representations.
 In Section~\ref{S:ProofTheos}, we give the proof of
Theorems~\ref{Theo1} and \ref{Theo2}. Section~\ref{S:ProofCor}
contains the proof of the corollaries.
Section~\ref{S:Examples} is devoted to the explicit computation of the Bohr compactification
and profinite completions for two groups:  the lamplighter group 
$(\ZZ/n\ZZ) \wr \ZZ$ and the Heisenberg group $H(R)$ over an arbitrary commutative ring $R.$

\section{Preliminaries}
\label{S:General}
 \subsection{Models for Bohr compactifications and profinite completions}
 Let $G$ be a topological group. We give  well-known models for $\Bohr(G)$ and $\Prof(G).$
 For this, we  use finite dimensional  unitary representations of $G,$ 
 that is,   continuous homomorphisms $\pi: G\to U(n)$ for some integer $n\geq 1.$
  We denote by $\widehat{G}_{\rm fd}$ the set   of  equivalence  classes of irreducible 
   finite dimensional  unitary representations of $G.$
   Let  $\widehat{G}_{\rm finite}$ be 
   the subset of  $\widehat{G}_{\rm fd}$ consisting of 
   representations $\pi$ with finite image $\pi(G).$
   
  For a compact (respectively, profinite) group $K$, the set $\widehat{K}_{\rm fd}$
  (respectively, $\widehat{K}_{\rm finite}$)  coincides with the dual space $\widehat{K}$,
  that is,  the set of equivalence classes  of  unitary representations of $K.$
 
 A useful tool for the identification of $\Bohr(G)$ or $\Prof(G)$ is  given by the following proposition; for the easy proof,
see Propositions 5 and 6 in  \cite{Bekka1}.
 \begin{proposition}
 \label{GeneralBohr1}
\begin{itemize}
 \item[(i)]  Let $K$  be a compact group and $\beta: G\to K$ a continuous homomorphism 
 with dense image; then $(K, \beta)$ is a Bohr compactification of $G$
 if and only if the map $\widehat{\beta}:\widehat{K}\to \widehat{G}_{\rm fd},$ given by 
 $\widehat{\beta}(\pi)= \pi\circ \beta,$ is surjective.
 \item[(ii)] Let $L$ a be profinite group and $\alpha: G\to L$ a continuous homomorphism 
 with dense image; then $(L, \alpha)$ is a profinite completion of $G$
 if and only if the map $\widehat{\beta}:\widehat{L}\to \widehat{G}_{\rm finite},$ given by 
 $\widehat{\beta}(\pi)= \pi\circ \beta,$ is surjective.
  \end{itemize}
 \end{proposition}

The following proposition is an immediate consequence of Proposition~\ref{GeneralBohr1}.
\begin{proposition}
\label{Pro-General2}
Choose  families 
$$(\pi_i: G\to U(n_i))_{i\in I} \quad \text{and} \quad (\sigma_j: G\to U(n_j))_{j\in J}$$
  of representatives for the  sets $\widehat{G}_{\rm fd}$ and $\widehat{G}_{\rm finite},$ respectively.
\begin{itemize}
\item[(i)] Let  $\beta: G \to \prod_{i \in I} U(n_i)$ be given by $\beta(g)= \bigoplus_{i \in I} \pi_i(g) $
and let $K$ be the closure of $\beta(G)$. Then $(K, \beta)$ is a Bohr compactification of $G.$
\item[(ii)]  Let  $\alpha: G \to \prod_{j \in J} U(n_j)$ be given by $\alpha(g)= \bigoplus_{j \in J} \sigma_j(g) $
and let $L$ be the closure of $\alpha(G)$. Then $(L, \alpha)$ is a profinite completion of $G.$
\end{itemize}
\end{proposition}
 
We observe that a more common for a profinite completion of $G$ is  the projective limit
   $\varprojlim G/H$, where $H$ runs over the  family of the normal subgroups of finite index of $G,$
together with the natural homomorphism $G \to \varprojlim G/H$ (see e.g. \cite[2.1.6]{Ribes-Zal})
\subsection{Extension of representations}
\label{SS:Extension}
We will also use the notion of a projective  representation.
Let $G$ be a locally compact group.  A  map $\pi: G \to U(n)$
is  a \textbf{projective representation} of $G$ if the following holds:
\begin{itemize}
 \item $\pi(e)=I$,
\item for all $g_1,g_2\in G,$ there exists  $c(g_1 , g_2 )\in {\mathbf S}^1 $
such that  $$\pi(g_1 g_2 ) = c(g_1 , g_2 )\pi(g_1 )\pi(g_2 ),$$
\item $\pi$ is  measurable.
\end{itemize}
The map  $c:G \times G \to {\mathbf S}^1$ is a $2$-cocycle with values in the unit cercle ${\mathbf S}^1.$
The conjugate representation $\overline{\pi}: G\to U(n)$ is 
another projective representation defined by $\overline{\pi}(g)= J\pi(g) J,$
where $J: \CCC^n\to \CCC^n$ is the anti-linear map given by conjugation of the coordinates,

The proof of the following  lemma  is straightforward.
\begin{lemma}
\label{Lem-TensorConj}
Let $\pi: G\to U(n)$ be a projective representation of $G,$ with associated  cocycle
$c:G \times G \to {\mathbf S}^1$.
Let  $\pi': G\to U(m)$ be another projective representation of $G$
with  associated  cocycle $2$-cocycle $c':G \times G \to {\mathbf S}^1$.
\begin{itemize}
\item[(i)] $\overline{\pi}: G\to U(n)$ is  a projective representation  of $G$ with  $\overline{c}$ as  associated  cocycle.
\item[(ii)] The tensor product 
$$\pi\otimes \pi': G\to U(nm), \qquad g\mapsto \pi(g)\otimes \pi'(g)$$ is  a projective representation of $G$ with $cc'$
as  associated  cocycle.
\end{itemize}
\end{lemma}

Let $N$ be a closed normal subgroup of $G.$  Recall that the stabilizer ${G}_\pi$ in $G$ of  an irreducible unitary representation $\pi$ of $N$ is the set of $g\in G$ such that  $\pi^g$ is equivalent to $\pi.$
Observe that  $G_\pi$ contains $N.$

The following proposition is a well-known  fact from the  Clifford-Mackey 
theory  of unitary representations of group extensions (see  \cite[Chap.1,  \S 11]{Curtis-Reiner} and \cite{Mackey-Acta}).
\begin{proposition}
\label{Pro-Ext} 
Let $G= N\rtimes H$ be the semi-direct product of the locally compact groups $H$ and $N$.
Let $\pi: N\to U(m)$ be an irreducible unitary representation $\pi$ of $N$
and assume that $G= G_\pi.$
There exists a projective representation $\widetilde\pi: G\to U(m)$
with the following properties:
\begin{itemize}
\item  $\widetilde\pi$ extends $\pi$, that is, $\widetilde\pi(n)= \pi(n)$ for 
every $n\in N;$
\item  the  $2$-cocycle $\widetilde{c}:G \times G\to {\mathbf S}^1$ associated to 
$\widetilde\pi$ has the form  $\widetilde{c}=c\circ (p\times p),$
for a map $c: H\times H\to {\mathbf S}^1$, where $p: G \to H$ is the canonical
homomorphism.
\end{itemize}
\end{proposition}
\begin{proof} Let $S\subset U(m)$ be a Borel transversal for 
the quotient space $PU(m)= U(m)/\mathbf{S}^1$
with $I_n\in S.$
Let $h\in H$. Since $G=G_\pi$ and since $\pi$ is irreducible,
 there exists a unique    matrix $\widetilde\pi(h)\in S$  such that
$$
\pi (hn h^{-1})= \widetilde\pi(h) \pi(n) \widetilde\pi(h)^{-1} \tout n\in N.
$$
Define $\widetilde\pi: G\to U(n)$ by 
$$
\widetilde\pi(nh)= \pi(n)\widetilde\pi(h)  \tout  n\in N, h\in H.
$$
It is clear that $\widetilde{\pi}|_N= \pi$ and that
$$\pi (gng^{-1})= \widetilde\pi(g) \pi(n) \widetilde\pi(g)^{-1} \tout g\in G, n\in N.$$
It can be shown (see proof of Theorem 8.2 in \cite{Mackey-Acta})
that $\widetilde{\pi}$ is a measurable map.

Let $g_1,g_2\in G$. For every $n\in N,$ we have, on the one hand,
$$
\pi (g_1g_2n g_2^{-1}g_1)=\widetilde\pi(g_1 g_2)\pi(n) \widetilde\pi(g_1 g_2)^{-1}
$$
and on the other hand
\[
\begin{aligned}
\pi (g_1g_2n g_2^{-1}g_1)&= \widetilde\pi(g_1) \pi(g_2 n g_2^{-1}) \widetilde\pi(g_1)^{-1}\\
&=\widetilde\pi(g_1) \widetilde\pi(g_2)\pi(n)\widetilde\pi (g_1)^{-1} \widetilde\pi (g_2)^{-1}.
\end{aligned}
\]
Since  $\pi$ is irreducible,
it follows that 
$$\widetilde\pi(g_1 g_2)=\widetilde{c}(g_1, g_2) \widetilde\pi (g_1)\widetilde\pi(g_2)$$
for some scalar $\widetilde{c}(g_1, g_2)\in {\mathbf S}^1$.

Moreover, for $g_1= n_1h_1, g_2= n_2 h_2,$ we have, on the one hand,
\[
 \begin{aligned}
 \widetilde\pi(g_1 g_2)&= \widetilde{c}(g_1, g_2) \widetilde\pi (g_1)\widetilde\pi(g_2)\\
 &= \widetilde{c}(n_2 h_1, n_2 h_1) \pi(n_1)\widetilde\pi(h_1) \pi(n_2)\widetilde\pi(h_2)
 \end{aligned}
 \]
  and, on the other hand,
 \[
 \begin{aligned}
  \widetilde\pi(g_1 g_2)&=\widetilde\pi(n_1(h_1n_2 h_1^{-1}) h_1h_2)\\
 & =\pi(n_1(h_1n_2 h_1^{-1})) \widetilde \pi(h_1h_2)\\
  &=\pi(n_1) \pi(h_1n_2 h_1^{-1}) \widetilde \pi(h_1h_2)\\
  &=\pi(n_1) \widetilde\pi(h_1) \pi(n_2) \widetilde\pi(h_1)^{-1} \widetilde \pi(h_1h_2)\\
  &=\widetilde{c}(h_1, h_2) \pi(n_1) \widetilde\pi(h_1) \pi(n_2) \widetilde\pi(h_1)^{-1} \widetilde \pi(h_1)\widetilde \pi(h_2)\\
  &=\widetilde{c}(h_1, h_2)\pi(n_1) \widetilde\pi(h_1) \pi(n_2)\widetilde \pi(h_2);
 \end{aligned}
 \]
 this shows  that $\widetilde{c}(n_2 h_1, n_2 h_1)=\widetilde{c}(h_1, h_2).$
\end{proof}

\subsection{ Bohr compactification and profinite completion of quotients}
\label{SS:BohrProf}
Let $G$ be a topological group and $N$ a closed normal subgroup of $G.$
Let $(\Bohr(G), \beta_G)$ and  $(\Prof(G), \alpha_G)$ be a Bohr compactification 
and a profinite completion of $G.$ 
Let $\Bohr(p): \Bohr(G)\to \Bohr(G/N)$ and  $\Prof(p): \Bohr(G)\to \Bohr(G/N)$
be the morphisms induced by  the canonical epimorphism $p: G\to G/N$.
The  following proposition is well-known (see \cite[Lemma 2.2]{Hart-Kunen} or  \cite[Proposition 10]{Bekka1} 
for (i)  and  \cite[Proposition 3.2.5]{Ribes-Zal} for (ii)).
For the convenience of the reader, we give  for (ii) a proof which is different from the one in \cite{Ribes-Zal}
\begin{proposition}
\label{Pro-Rightexact}
\label{Pro-BohrQuotient}
\begin{itemize}
\item[(i)] $\Bohr(p)$ is surjective and its kernel is  $\overline{\beta_G(N)}$.
\item[(ii)] $\Prof(p)$ is surjective and its kernel is  $\overline{\alpha_G(N)}$
\end{itemize}
\end{proposition}
\begin{proof}
To show (ii),  set $K:=\overline{\alpha_G}(N)$.
Let   $(\Prof(G/N),\overline{\alpha})$ be a Bohr compactification of $G/N.$
We have a commutative  diagram 
\[
\begin{tikzcd}
G\arrow{r}{p}\arrow{d}{\alpha_G}&G/N  \arrow{d}{\overline{\alpha}}\\
\Prof(G)  \arrow{r}{\Prof(p)} & \Prof(G/N)
\end{tikzcd}
\]
It follows that $\alpha_G(N)$ and hence $K$ is contained in $\Ker(\Prof(p)).$
So, we have  induced homomorphisms $\beta: G/N\to   \Prof(G)/K$
and $\beta': \Prof(G)/K\to \Prof(G/N),$
giving rise to a  commutative diagram
\[
\begin{tikzcd}
&G/N  \arrow{dl}{\beta}\arrow{d}{\overline{\alpha}} \\
 \Prof(G)/K \arrow{r}{\beta'} & \Prof(G/N) \\
\end{tikzcd}
\] 
It follows that $(\Prof(G)/K, \beta)$ has the same  universal property 
for $G/N$ as $(\Prof(G/N),\overline{\alpha})$; it is therefore a profinite completion of $G/N.$
\end{proof}

 \section{Proof of Theorems~\ref{Theo1} and \ref{Theo2}}
 \label{S:ProofTheos}
 \subsection{Proof of Theorem~\ref{Theo1}}
 \label{SS:ProofTheo1}
  Set $K:= \overline{\beta_G(N)},$ where $\beta_G$ is the canonical map 
 from the locally compact group $G= N\rtimes H$ to $\Bohr(G).$

 \vskip.2cm
  $\bullet$ {\it First step.} We claim that 
  $$
 \left\{\widehat{\sigma}\circ (\beta_G|_N): \widehat{\sigma} \in \widehat{K}\right\} \subset \widehat{N}_{ \rm fd}^{H-{\rm per}}.
  $$
  Indeed, let  $\widehat{\sigma} \in \widehat{K}$.
Then  $\sigma:= \widehat{\sigma}\circ (\beta_G|_N)\in \widehat{N}_{ \rm fd}.$
Let $ \widehat{\rho}\in \widehat{\Bohr(G)}$ be an irreducible subrepresentation of the induced representation 
  $\Ind_{K}^{\Bohr(G)} \widehat{\sigma}.$ Then, by Frobenius reciprocity, $\widehat{\sigma}$ is equivalent to a subrepresentation
  of $\widehat{\rho}|_K.$ Hence, $\sigma$ is equivalent to a subrepresentation of 
   $(\widehat{\rho} \circ \beta_G)|N.$
  The decomposition of  the finite dimensional representation $(\widehat{\rho} \circ \beta_G)|_N$ into isotypical components
  shows that $\sigma$ has a finite $H$-orbit (see \cite[Proposition 12]{Bekka1}).
  
  \vskip.2cm
  $\bullet$ {\it Second step.} 
  We claim that 
  $$
 \widehat{N}_{ \rm fd}^{H-{\rm per}} \subset \left\{\widehat{\sigma} \circ (\beta_G|_N): \widehat{\sigma} \in \widehat{K}\right\}.
  $$
Indeed, let  $\sigma: N\to U(m)$ be a  representation of $N$ with finite $H$-orbit. 
 By Proposition~\ref{Pro-Ext}, there exists a \emph{projective} representation $\widetilde{\sigma}$ of $G_\sigma=NH_\sigma$ which extends $\sigma$ and  the associated cocycle $c: G_\sigma\times G_\sigma\to \mathbf{S}^1$,
 factorizes through $H_\sigma\times H_\sigma$.
 
 Define a projective representation  $\tau: G_\sigma\to U(m)$ of  $G_\sigma$ by 
 $$ 
 \tau(hn)=  \overline{\widetilde{\sigma}}(h) \tout nh\in NH_\sigma.
 $$
 Observe that $\tau$ is  trivial on $N$ and that its associated cocycle is $\overline{c}.$ 
Consider the  tensor product representation   $\widetilde{\sigma}\otimes \tau$ of $G_\sigma.$ 
 Lemma~\ref{Lem-TensorConj} shows that $\widetilde{\sigma}\otimes \tau$ is a projective representation for the  cocyle
 $c\overline{c}=1.$ So, $\widetilde{\sigma}\otimes \tau$ is a measurable homomorphism
 from $G_\sigma\to U(m).$ This implies that $\widetilde{\sigma}\otimes \tau$ is continuous 
 (see \cite[Lemma A.6.2]{BHV}) and so $\widetilde{\sigma}\otimes \tau$ is
 an \emph{ordinary}  representation of $G_\sigma$.

  It is clear that $\widetilde{\sigma}\otimes \tau$ is finite dimensional. Observe that   the restriction $(\widetilde{\sigma}\otimes \tau)|_N$ of $\widetilde{\sigma}\otimes \tau$ to $N$ is a multiple of $\sigma.$
  Let 
  $$\rho:=\Ind_{G_\sigma}^G (\widetilde{\sigma}\otimes \tau).$$
  Then $\rho$  is  finite dimensional, since $\widetilde{\sigma}\otimes \tau$ is finite dimensional
and $G_\sigma$ has finite index in $G$.
As $G_\sigma$ is open in $G,$ $\widetilde{\sigma}\otimes \tau$ is equivalent to a subrepresentation of the restriction  $\rho|_{G_\sigma}$ of $\rho$ to $G_\sigma$ (see e.g. \cite[1.F]{BH}); consequently, $\sigma$ is equivalent to a subrepresentation
of $\rho|_{N}$. Since $\rho$ is a finite dimensional unitary representation of $G,$
there exists a unitary representation $\widehat{\rho}$ of $\Bohr(G)$
such that $\widehat{\rho}\circ \beta_G= \rho.$
So, $\sigma$ is equivalent to a subrepresentation
of $(\widehat{\rho}\circ \beta_G)|_N$, that is, there exists
a   subspace $V$ of the space of $\widehat{\rho}$ which is invariant under $\beta_G(N)$
and defining a representation of $N$ which is equivalent to $\sigma.$
Then $V$ is invariant under $K=\overline{\beta_G(N)}$ and defines therefore
an irreducible  representation $\widehat{\sigma}$ of $K$ for which
$\widehat{\sigma}\circ (\beta_G|_N)= \sigma$ holds.

\vskip.2cm
Let $\vfi_N: \Bohr(N)\to K=\overline{\beta_G(N)}$ be the homomorphism such that $\vfi_N\circ \beta_N= \beta_G|_N$.

\vskip.2cm
  $\bullet$ {\it Third step.} We claim that 
$$
\Ker \vfi_N =\bigcap_{\sigma \in \widehat{N}_{ \rm fd}^{H-{\rm per}}} \Ker (\Bohr({\sigma})),
$$
where  $\Bohr({\sigma})$ is the representation of $\Bohr(N)$ such that  $\Bohr({\sigma})\circ \beta_N= \sigma.$

Indeed, by the first and second steps, we have
$$
\widehat{N}_{ \rm fd}^{H-{\rm per}}=\left\{\widehat{\sigma} \circ (\beta_G|_N): \widehat{\sigma} \in \widehat{K}\right\}=
\left\{(\widehat{\sigma} \circ \vfi_N) \circ \beta_N: \widehat{\sigma} \in \widehat{K}\right\};
$$
since obviously $\widehat{\sigma} \circ \vfi_N= \Bohr({\sigma})$ for $\sigma= \widehat{\sigma} \circ \vfi_N,$
it follows that
\[
\begin{aligned}
\bigcap_{\sigma \in \widehat{N}_{ \rm fd}^{H-{\rm per}}} \Ker (\Bohr({\sigma}))&=
 \bigcap_{\widehat{\sigma} \in \widehat{K}} \Ker (\widehat{\sigma} \circ \vfi_N).
 \end{aligned}
 \]
As   $\vfi_N(\Bohr(N))=K$ and $\widehat{K}$ separates the points of $K,$ we have 
$\bigcap_{\widehat{\sigma} \in \widehat{K}} \Ker (\widehat{\sigma} \circ \vfi_N)=\Ker \vfi_N$
and the claim is proved.

\vskip.2cm

Set $L:=  \overline{\beta_G(H)}.$

  $\bullet$ {\it Fourth step.} We claim that the map $\vfi_H: \Bohr(H)\to L,$
  defined by the relation $\vfi_H\circ \beta_H= \beta_G|_H,$ is an isomorphism.
  Indeed, the canonical isomorphism $H\to G/N$ induces an isomorphism $\Bohr(H) \to \Bohr(G/N)$.
  Using  Proposition~\ref{Pro-Rightexact}.i., we obtain  a continuous epimorphism  
  $$f:L\to \Bohr(H)$$
  such that  $f(\beta_G(h))= \beta_H(h)$ for all $h\in H.$
  Then $\vfi_H\circ f$ is the identity on $\beta_G(H)$ and hence on $L,$ by density. This implies that 
  $f$ is an isomorphism.

\vskip.2cm
Observe   that, by the universal property of $\Bohr(N),$
  every element $h\in H$ defines   a continuous automorphism $\theta_b(h)$ of $\Bohr(N) $ 
 such that 
 $$\theta_b(h)(n)= \beta_N(hnh^{-1})\tout n\in N.$$ 
   The  corresponding homomorphism $\theta_b: H\to  \Aut(\Bohr(N))$
   defines an action of $H$ on the compact group $\Bohr(N).$
   By duality, we have an action, still denoted by $\theta_b,$ of $H$
   on  $\widehat{\Bohr(N)}$
 and we have 
   $$
   \Bohr(\sigma^h)= \theta_b(h)(\Bohr(\sigma))  \tout \sigma \in \widehat{N}_{ \rm fd}, h\in H.
   $$
 This implies that the normal subgroup
   $$
\Ker \vfi_N=\bigcap_{\sigma \in \widehat{N}_{ \rm fd}^{H-{\rm per}}} \Ker (\Bohr(\sigma)).
$$
of $\Bohr(N)$ is $H$-invariant. We have therefore an induced action $\overline{\theta_b}$ of $H$ on 
$\Bohr(N)/\Ker\vfi_N.$
Observe that the isomorphism 
$$\Bohr(N)/\Ker\vfi_N\to K$$ 
induced by $\vfi_N$ is $H$-equivariant for $\overline{\theta_b}$ and the action
of $H$ on $K$ given by conjugation with $\beta_G(h)$ for $h\in H.$

$\bullet$ {\it Fifth step.} We claim that   the action $\overline{\theta_b}$  induces an action of 
$\Bohr(H)$ by automorphisms on  $\Bohr(N)/\Ker \vfi_N$ and that the map
$$(\Bohr(N)/\Ker\vfi_N) \rtimes \Bohr(H)\to \Bohr(G), (x \Ker\vfi_N, y) \mapsto \vfi_N(x)\vfi_H(y)$$
is an isomorphism.

Indeed, $\overline{\beta_G(N)}$ is a normal subgroup of $\Bohr(G)$
and so  $\overline{\beta_G(H)}$ acts by conjugation  on $K.$
By the third and the forth step, the maps   
$$\overline{\vfi_N}: \Bohr(N)/\Ker \vfi_N  \to K, \qquad x \Ker\vfi_N\mapsto \vfi_N(x)$$ 
and 
$$ \vfi_H:\Bohr(H)\to L$$
are  isomorphisms. We define an action 
$$\widehat{\theta}: \Bohr(H) \to \Aut(\Bohr(N)/\Ker \vfi_N)$$
by $$\widehat{\theta}(y)(x \Ker\vfi_N)= (\overline{\vfi_N})^{-1} \left(\vfi_H(y) \vfi_N(x) \vfi_H(y)^{-1}\right)$$
for $x\in \Bohr(N)$ and  $y\in \Bohr(H).$
The claim follows.

 \subsection{Proof of Theorem~\ref{Theo2}}
 The proof is similar to the proof of  Theorem~\ref{Theo1}. The role of   $\widehat{N}_{ \rm fd}$ is now played by  the space $\widehat{N}_{ \rm finite}$ of finite dimensional irreducible representations of $N$ with finite image. 
 We will go quickly through the steps of the proof of  Theorem~\ref{Theo1}; at some places (especially the second step)
 there will be a few crucial changes and new arguments which we will emphasize.

 Set $L:= \overline{\alpha_G(N)},$ where $\alpha_G: G \to \Prof(G)$
 is the canonical map.  Observe that $L$ is profinite.
 
 \vskip.2cm
  $\bullet$ {\it First step.} We claim that $
 \left\{\widehat{\sigma}\circ (\alpha_G|_N): \widehat{\sigma} \in \widehat{L}\right\} \subset \widehat{N}_{ \rm finite}^{H-{\rm per}}.
  $
  Indeed, let  $\widehat{\sigma} \in \widehat{L}$.
Then  $\sigma:= \widehat{\sigma}\circ (\alpha_G|_N)\in \widehat{N}_{ \rm finite},$
since $L$ is profinite.
Let $ \widehat{\rho}$ be an irreducible subrepresentation of 
  $\Ind_{L}^{\Prof(G)} \widehat{\sigma}.$ Since $\Prof(G)$ is compact, $\widehat{\rho}$ 
  is finite dimensional. Since $\sigma$ is equivalent to a subrepresentation of 
   $\widehat{\rho} \circ (\alpha_G)|N)$, it has therefore a finite $H$-orbit.

   \vskip.2cm
  $\bullet$ {\it Second step.} 
  We claim that $
 \widehat{N}_{ \rm finite}^{H-{\rm per}} \subset \left\{\widehat{\sigma} \circ (\alpha_G|_N): \widehat{\sigma} \in \widehat{L}\right\}.
  $
Indeed, let  $\sigma: N\to U(m)$ be an irreducible representation with finite image.
 By Proposition~\ref{Pro-Ext}, there exists a projective representation $\widetilde{\sigma}$ of $G_\sigma=NH_\sigma$ which extends $\sigma$ and  the associated cocycle $c: G_\sigma\times G_\sigma\to \mathbf{S}^1$,
 factorizes through $H_\sigma\times H_\sigma$.
 We need to show that we can choose $\widetilde{\sigma}$ so that $\widetilde{\sigma}(G_\sigma)$ is finite. 
 
 Choose a  projective representation  $\widetilde{\sigma}: G_\sigma \to U(m)$ as above and modify $\widetilde{\sigma}$  as follows:  define
 $$\widetilde{\sigma}_1(hn)= \dfrac{1}{(\det \widetilde{\sigma}(h))^{1/m}} \widetilde{\sigma}(h) \sigma(n)
 \tout  h\in H_\sigma, n\in H.
 $$
 Then $\widetilde{\sigma}_1$ is again a projective representation of  $G_\sigma=NH_\sigma$ which extends $\sigma$ and  the associated cocycle $c: G_\sigma\times G_\sigma\to \mathbf{S}^1$ factorizes through $H_\sigma\times H_\sigma$;
 moreover, $\widetilde{\sigma}_1(h)\in SU(m)$ for every $h\in H_\sigma.$
 
 Every $h\in  H_\sigma$  induces  a  bijection $\vfi_h$ of  $\sigma(N)$
 given by 
 $$\vfi_h: \sigma(n) \mapsto \widetilde{\sigma}_1(h)\sigma(n) \widetilde{\sigma}_1(h)^{-1}= \sigma (h nh^{-1}) \tout n\in N.$$
  So, we have a map
 $$
\vfi:\widetilde{\sigma}_1(H_\sigma)\to\Sym(\sigma(N)), \quad \widetilde{\sigma}_1(h)\mapsto\vfi_h
$$
where $\Sym(\sigma(N))$ is the set of bijections of $\sigma(N).$
 For $h_1, h_2\in H_\sigma,$ we have $\vfi_{h_1}= \vfi_{h_2}$ if and only if 
$\widetilde{\sigma}_1(h_2)= \lambda \widetilde{\sigma}_1(h_1)$ for some scalar
$ \lambda \in \mathbf{S}^1,$ by irreducibility of $\sigma$.
Since $\det(\widetilde{\sigma}_1(h_1))=1$ and  $\det (\widetilde{\sigma}_1(h_2))=1,$
it follows that $\la$ is a $m$-th root of unity. 
This shows that the fibers of the  map $\vfi$ are finite. 
Since $\sigma(N)$ is finite, $\Sym(\sigma(N))$ and  hence $\widetilde{\sigma}_1(H_\sigma)$  is finite.
It follows that $\widetilde{\sigma}_1(G_\sigma)= \widetilde{\sigma}_1(H_\sigma)\sigma(N)$ is finite.
 
Let  $\tau: G_\sigma\to U(m)$ be the projective representation of  $G_\sigma$ given by 
 $$ 
 \tau(hn)=  \overline{\widetilde{\sigma}_1}(h) \tout nh\in NH_\sigma.
 $$
Then $\widetilde{\sigma}_1\otimes \tau$ is a ordinary  representation of $G_\sigma$
and has finite image.
  The induced representation $\rho:=\Ind_{G_\sigma}^G (\widetilde{\sigma}_1\otimes \tau)$
  has finite image,  since $G_\sigma$ has finite index in $G$.
As $\widetilde{\sigma}_1\otimes \tau$ is equivalent to a subrepresentation of the restriction  $\rho|_{G_\sigma}$
of $\rho$ to $G_\sigma,$ the representation
 $\sigma$ is equivalent to a subrepresentation
of $\rho|_{N}$. Since $\rho(G)$ has finite image,
there exists a unitary representation $\widehat{\rho}$ of $\Prof(G)$
such that $\widehat{\rho}\circ \alpha_G= \rho.$
So, there exists
a   subspace $V$ of the space of $\widehat{\rho}$ which is invariant under $\alpha_G(N)$
and defining a representation of $N$ which is equivalent to $\sigma.$
Then $V$  defines 
an irreducible  representation $\widehat{\sigma}$ of $L$ for which
$\widehat{\sigma}\circ (\alpha_G|_N)= \sigma$ holds.

\vskip.2cm
Let $\psi_N: \Prof(N)\to L$ be the homomorphism such that $\psi_N\circ \alpha_N= \alpha_G|_N$.

\vskip.2cm
  $\bullet$ {\it Third step.} We claim that 
$$
\Ker \psi_N =\bigcap_{\sigma \in \widehat{N}_{ \rm finite}^{H-{\rm per}}} \Ker (\Prof(\sigma)).
$$
Indeed, the proof is similar to the proof of the third step of Theorem~\ref{Theo1}

\vskip.2cm
  $\bullet$ {\it Fourth step.} We claim that the map $\psi_H: \Prof(H)\to \overline{\alpha_G(H)},$
  defined by the relation $\vfi_H\circ \alpha_H= \alpha_G|_H,$ is an isomorphism.
  Indeed,  the proof is similar to the proof of the fourth step  of Theorem~\ref{Theo1}.
   
\vskip.2cm
Every element $h\in H$ defines   a continuous automorphism $\theta_p(h)$ of $\Prof(N).$ 
 Let  $$\theta_p: H\to  \Aut(\Prof(N))$$ be the corresponding homomorphism;
 as in Theorem~\ref{Theo1}, we have an  induced action $\overline{\theta_p}$ of $H$ on 
$\Prof(N)/\Ker\psi_N.$

$\bullet$ {\it Fifth step.} We claim that   the action $\overline{\theta_p}$ of $H$  induces an action of 
$\Prof(H)$ by automorphisms on  $\Prof(N)/\Ker \psi_N$ and that the map
$$\left(\Prof(N)/\Ker\psi_N\right) \rtimes \Prof(H)\to \Prof(G), (x \Ker\psi_N, y) \mapsto \psi_N(x)\psi_H(y)$$
is an isomorphism.

Indeed,  the proof is similar to the proof of the fifth step  of Theorem~\ref{Theo1}.  

 \label{SS:ProofTheo2}
 \section{Proof of the Corollaries}
 \label{S:ProofCor}
 \subsection{Proof of Corollary~\ref{Cor0}}
 Assume that $N$ is finitely generated. In view of Theorem~\ref{Theo2},
 we have to show that  $\widehat{N}_{ \rm finite}^{H-{\rm per}}=  \widehat{N}_{ \rm finite}.$
 
 It is well-known that, for every integer $n\geq1,$
 there are only finitely many subgroups of index $n$ in $N.$ Indeed, since $N$ is finitely generated,
 there are only finitely many actions of $N$ on the set $\{1, \dots, n\}.$
Every subgroup $M$ of index $n$ defines an action of $N$ on $N/M$ and hence
on $\{1, \dots, n\}$ for which the stabilizer of, say, 1 is $M.$
So, there  are only finitely many such subgroups $M.$

Let $\sigma \in \widehat{N}_{ \rm finite}$ and set $n:=  |\sigma(N)|.$
Consider $N_\sigma= \cap_{ M} M,$
where $M$ runs over the subgroups of $N$ of index $n.$
Then $N_\sigma$ is a normal subgroup of $N$ of finite index and, for every $h\in H,$
the representation  $\sigma^h$ factorizes to a representation of  $N/N_\sigma.$
Since $N/N_\sigma$ is a  finite group, it has only finitely many non equivalent irreducible representations and the
claim is proved.
 \subsection{Proof of Corollary~\ref{Cor1}}
 We assume that $N$ is abelian.
 The dual group of $\Bohr(N)$ is $\widehat{N}$  and the dual of $\Prof(N)$ is  
 $\widehat{N}_{\rm fin}$, viewed as discrete groups.
 With the notation as in Theorems~\ref{Theo1} and ~\ref{Theo2}, the subgroups
 $C$ and $D$ are respectively the annihilators in $\Bohr(N)$ and in $\Prof(N)$ of the closed 
 subgroups  $\widehat{N}^{H-{\rm per}}$ and  $\widehat{N}_{ \rm finite}^{H-{\rm per}}$.
 Hence,  $\Bohr(N)/C$ and $\Prof(N)/D$
 are the dual groups of $\widehat{N}^{H-{\rm per}}$ and  $\widehat{N}_{ \rm finite}^{H-{\rm per}}$,
 viewed as discrete groups.
So, the claim follows from Theorems~\ref{Theo1} and ~\ref{Theo2}.

 \subsection{Proof of Corollary~\ref{Cor2}}
 In view of  Theorems~\ref{Theo1} and ~\ref{Theo2},
  $G$ is MAP, respectively RF, if and only if
 $$\Ker (\vfi_N \circ \beta_N)=\{e\} \qquad \text{and} \qquad \Ker (\vfi_H \circ \beta_H)=\{e\},$$
 respectively
 $$\Ker (\psi_N \circ \alpha_N)=\{e\}=\{e\} \qquad \text{and} \qquad \Ker (\psi_H \circ \alpha_H)=\{e\}.$$
 So, 
  $G$ is MAP, respectively RF, if and only if
 $$\beta_N^{-1}(C)=\{e\} \qquad \text{and} \qquad \Ker ( \beta_H)=\{e\},$$
 respectively
 $$\alpha_N^{-1}(D)=\{e\} \qquad \text{and} \qquad \Ker (\alpha_H)=\{e\}.$$
 This exactly means that $G$ is MAP, respectively RF, if and only if
 $\widehat{N}_{ \rm fd}^{H-{\rm per}}$ separates the points of $N$ and $H$ is MAP,
 respectively $\widehat{N}_{ \rm finite}^{H-{\rm per}}$ separates the points of $N$ and $H$ is RF.

  \subsection{Proof of Corollary~\ref{Cor3}}
 We assume that  $G= \Lambda\wr H$ is the wreath product of the countable groups  $\Lambda$ and $H$
 and set $N:=\oplus_{h\in H}\Lambda.$

\noindent
\vskip.2cm
\noindent 
(i) Assume that $H$ is finite. Then, of course, 
$\widehat{N}_{\rm fd}^{H-{\rm per}}=\widehat{N}_{\rm fd}$ and $\widehat{N}_{ \rm finite}^{H-{\rm per}}
=\widehat{N}_{\rm finite};$
so, the subgroups  $C$ and $D$ from Theorems~\ref{Theo1} and ~\ref{Theo2} are trivial. 
Since $\Bohr(N)=\oplus_{h\in H}\Bohr(\Lambda)$ and $ \Prof(N)= \oplus_{h\in H}\Prof(\Lambda),$
we have
\[\begin{aligned}
&\Bohr(\Lambda\wr H)\cong \left( \oplus_{h\in H}\Bohr(\Lambda)\right) \rtimes \Bohr(H) \ \text{and}\\
&\Prof(\Lambda\wr H)\cong \left( \oplus_{h\in H}\Prof(\Lambda)\right) \rtimes \Prof(H).
\end{aligned}
\]

\noindent
\vskip.2cm 
\noindent 
(ii) Assume that $H$ is infinite.

\vskip.2cm
  $\bullet$ {\it First step.}  We claim that, for every $\sigma \in \widehat{N}_{\rm fd}^{H-{\rm per}},$  
  we  have $\dim \sigma =1$, that is, $\sigma (N)\subset U(1)=\mathbf{S}^1.$
  
Indeed, assume by contradiction that $\dim \sigma >1$.
Let $\mathcal{F}$ be the family of finite subsets
of $H.$ For every  $F\in \mathcal{F},$ let $N(F)$ be  the  normal subgroup
of $N$ given by 
 $$N(F):=\oplus_{h\in F}\Lambda$$
The restriction $\sigma|_{N(F)}$  of $\sigma$ to  $N(F)$  has a  decomposition into isotypical components:
 $$
 \sigma|_{N(F)} = \oplus_{\pi\in  \Sigma_F} n_\pi \pi,
 $$
 where $\Sigma_F$ is a (finite) subset of  $ \widehat{N(F)}_{\rm fd}$ 
 and the $n_\pi$'s some positive integers.
 As is well-known (see, e.g., \cite[\S 17]{Weil}), every representation in $\widehat{N(F)}_{\rm fd}$ 
 is a tensor product  $\otimes_{h\in F}\rho_h$ of irreducible representations $\rho_h$ of 
$\Lambda$ ; so, we can view $\Sigma_F$ as subset of $\prod_{h\in F}\widehat{\Lambda}_{\rm fd}.$ 
If $F\subset F',$ then the obvious map $\prod_{h\in F'}\widehat{\Lambda}_{\rm fd}\to \prod_{h\in F}\widehat{\Lambda}_{\rm fd}$
restricts to a surjective map $\Sigma_{F'} \to \Sigma_F$.

Since $\dim \sigma$ is finite, it follows that there exists $F_0\in \mathcal{F}$  such that
$$\dim \pi =1 \tout \pi\in \Sigma_F,  F\in \mathcal{F} \quad \text{with} \quad
F\cap F_0=\emptyset
$$
and 
$$\dim\pi_0>1 \quad \text{for some}\quad\pi_0\in\Sigma_{F_0}.$$
For $h\in H$ and $F\in \mathcal{F},$ 
observe that for the decomposition of $\sigma^h|_{N( h^{-1}F)}$ into isotypical 
components, we have
$$\sigma^{h}|_{N(h^{-1}F)}= \oplus_{\pi\in  \Sigma_{F}} n_\pi \pi.$$
So, $\sigma^{h}$ and $\sigma$ are not equivalent if 
$h^{-1}F_0\cap F_0=\emptyset.$

We choose  inductively a sequence $(h_n)_{n\geq 0}$ of elements in  $H$  by $h_0=e$ 
and 
$$h_{n+1}^{-1} F_{0} \cap \bigcup_{0\leq m\leq n} h_{m}^{-1} F_{0} =\emptyset \tout n \geq 0.$$
The $\sigma^{h_n}$'s are then pairwise not equivalent.
This is a contradiction, since  $\sigma \in \widehat{N}_{\rm fd}^{H-{\rm per}}.$

\vskip.3cm
Let $p:\Lambda\wr H\to \Lambda^{\rm Ab}\wr H$ be the quotient map, which is
given by 
$$
p((\lambda_a)_{a\in H},  h)= ((\lambda_a [\Lambda, \Lambda])_{a\in H},  h).
$$

\vskip.2cm
  $\bullet$ {\it Second step.}  
  We claim that the induced maps 
  $$\Bohr(p):  \Bohr(\Lambda\wr H)\to \Bohr( \Lambda^{\rm Ab}\wr H)$$
  and
  $$
  \Prof(p): \Prof(\Lambda\wr H)\to \Prof( \Lambda^{\rm Ab}\wr H)
  $$
are isomorphisms.

Indeed, by the first step, every $\sigma \in \widehat{N}_{\rm fd}^{H-{\rm per}}$ factorizes through 
$ N^{\rm Ab}.$ Hence, by Theorems~\ref{Theo1} and \ref{Theo2},
 $ [N, N]$ is contained in $C=\ker \vfi_N$ and $ [N, N]$ is contained in $D=\ker \psi_N$.
 This means that $\beta_G(\ker p)= \{e\}$ and  $\alpha_G(\ker p)= \{e\}$.
 The claim follows then from Proposition~\ref{Pro-BohrQuotient}.

 \subsection{Proof of Corollary~\ref{Cor4}}
 We assume that  $G$ is a wreath product 
 $G= \Lambda\wr H$
 and  that $\Lambda$ has at least two elements.
 As before, we set $N=\oplus_{h\in H}\Lambda .$

\vskip.2cm
 \noindent
 (i) Assume that $H$ is finite. Then $G$ is MAP (respectively RF)
if and only if  $\Lambda$ is MAP (respectively RF).

Indeed, $\widehat{N}_{ \rm fd}^{H-{\rm per}}= \widehat{N}_{ \rm fd}$
separates the points of $N.$ The claim follows then from  Corollary~\ref{Cor2}.

\vskip.2cm
 \noindent
(ii) Assume that $H$ is infinite.
If $G$ is MAP, then $\Lambda$ is abelian, by Corollary~\ref{Cor3}.
So, we may   and  will from now assume  that  $\Lambda$ (and hence $N$) is abelian.

 \vskip.2cm
  $\bullet$ {\it First step.}  We claim that, if  $\widehat{N}^{H-{\rm per}}$ separates the points of $N$,
then  $H$ is  RF.

Indeed, recall that the dual group $\widehat{\Lambda}$ of $\Lambda,$
equipped with the topology of pointwise convergence, is a compact group.
The dual group $\widehat{N}$ of $N$ can be identified with, as topological group,
with the  product  group $\prod_{h\in H}\widehat{\Lambda},$
endowed with the product topology, by means of the duality 
$$
\left \langle \prod_{h\in H}\chi_h, \oplus_{h\in H}\lambda_h\right\rangle= \prod_{h\in H}\chi_h(\lambda_h) 
\quad \text{for all} \quad 
\prod_{h\in H}\chi_h\in \widehat{N}, \oplus_{h\in H}\lambda_h\in N.
$$
 (Observe that the product on the right hand side   is well-defined
 since $\lambda_h=e$ for all but finitely many $h\in H.$)
 The dual action  of $H$ on $\widehat{N}$ is given by 
  $$ 
  (\prod_{h\in H}\chi_h)^a = \prod_{h\in H}\chi_{a^{-1}h} \tout a\in H.
 $$
 For $\Phi:=\prod_{h\in H}\chi_h \in \widehat{N}$, we have that 
 $\Phi\in \widehat{N}^{H-{\rm per}}$
  if and only if there exists a finite index subgroup $H_\Phi$ of $H$ such that 
 $$\chi_{ah}=\chi_h \tout a\in H_\Phi, h\in H.$$
  
  By assumption,  $\widehat{N}^{H-{\rm per}}$ separates the points of $N$ ; 
  equivalently,  $\widehat{N}^{H-{\rm per}}$  is dense in $\widehat{N}$.
  Let $h_0\in H\setminus \{e\}.$
  Since $\Lambda$ has at least two elements, we can find $\chi^0\in \widehat{\Lambda}$ 
  and $\lambda_0\in \Lambda$ with $\chi^0(\lambda_0)\neq 1.$
  Define $\Phi_0= \prod_{h\in H}\chi_h\in  \widehat{N}$
  by $\chi_{h_0}= \chi^0$ and $\chi_h= 1_{\Lambda}$
  for $h\neq h_0.$ Set 
 $$\eps:= \dfrac{1}{2} \left|\chi^0(\lambda_0)-1\right|>0.$$
 Since  $\widehat{N}^{H-{\rm per}}$  is dense in $\widehat{N}$,
 we can find $\Phi'= \prod_{h\in H}\chi_h'\in  \widehat{N}^{H-{\rm per}}$
 such that 
 $$
 | \chi_{h_0}'(\lambda_0)- \chi_{h_0}(\lambda_0)| \leq  \eps/2 
 \quad \text{and}\quad  | \chi_{e}'(\lambda_0)- \chi_{e}(\lambda_0)|\leq \eps/2
 $$
 We claim that $h_0$ does not belong to the stabilizer $H_{\Phi'}$ of $\Phi'.$
 Indeed, assume by contradiction that $h_0\in H_{\Phi'}.$
 Then $\chi_{h_0}'= \chi_{e}'$ and hence
 \[
\begin{aligned}
2\eps&=  |\chi^0(\lambda_0)-1|\\
&\leq |\chi^0(\lambda_0)-\chi_{h_0}'(\lambda_0)|+ |\chi_{h_0}'(\lambda_0)- 1|\\
 &=|\chi_{h_0}(\lambda_0)-\chi_{h_0}'(\lambda_0)|+  | \chi_{e}'(\lambda_0)- \chi_{e}(\lambda_0)| \\
 &\leq \eps
 \end{aligned}
\]
and this is a contradiction.
Since $H_{\Phi'}$ has finite index, we have proved that  $H$ is RF.

\vskip.2cm
  $\bullet$ {\it  Second  step.}  We claim that, if  $H$ is  RF, 
  then $\widehat{N}^{H-{\rm per}}$ separates the points of $N$.

Indeed, let $\oplus_{h\in H}\lambda_h\in N\setminus\{e\}.$
Then  $F=\{h\in H : \lambda_h\neq e\}$ is a finite and non-empty subset of $H.$
Let $(\chi^0_h)_{h\in F}$ be a sequence in $ \widehat{\Lambda}$   such that 
$\prod_{h\in F} \chi^0_h(\lambda_h) \neq 1$ (this is possible, since 
abelian groups are MAP).
Since $H$ is RF, we can find a subgroup of finite index $L$ of
$H$ so that $Lh \neq Lh'$ for all $h, h'\in F$ with $h\neq h'.$
Define $\Phi=  \prod_{h'\in H}\chi_{h'}\in  \widehat{N}$
by 
$$
\chi_{h'}=\begin{cases}
\chi^0_h& \text{if}\quad h'\in  Lh \quad \text{for some} \quad h\in F\\
1_{\Lambda}& \text{if}\quad  h'\notin \cup_{h\in F} Lh
\end{cases}
$$ 
It is clear that $L\subset H_{\Phi}$ and hence that 
$\Phi\in \widehat{N}^{H-{\rm per}}$; moreover,
$$\Phi\left(\oplus_{h\in H}\lambda_h \right)= \prod_{h\in F} \chi^0_h(\lambda_h) \neq 1.$$
So, $\widehat{N}^{H-{\rm per}}$ separates the points of $N.$

 \vskip.2cm
  $\bullet$ {\it  Third  step.}  We claim that, if  $H$ is  RF  and $\Lambda$ is RF,
  then $\widehat{N}_{\rm finite}^{H-{\rm per}}$ separates the points of $N$.
  
  The proof is  the same as the proof of the second step, with only one difference:
  one has to choose a sequence a sequence  $(\chi^0_h)_{h\in F}$  in 
  $\chi^0_h\in  \widehat{\Lambda}_{\rm finite}$ such that
  $\prod_{h\in F} \chi^0_h(\lambda_h) \neq 1$; this is possible, since we are assuming that  $\Lambda$ is RF.
 
 \vskip.2cm
  $\bullet$ {\it  Fourth  step.}  We claim that $G$ is MAP  if and only if  $H$ is  RF.
  Indeed, this follows from Corollary~\ref{Cor2}, combined with the first  and second steps.

 \vskip.2cm
  $\bullet$ {\it  Fifth  step.}  We claim that $G$ is RF  if and only if  $\Lambda$ and $H$ are  RF
  Indeed, this follows from Corollary~\ref{Cor2}, combined with the first and third steps.

 \section{Examples}
 \label{S:Examples}
 \subsection{Lamplighter group}
 \label{SS:Lamplighter}
 For $m\geq 1,$ denote by $C_m$  the finite cyclic group $\ZZ/m\ZZ.$
 Recall that 
 $$\Bohr(\ZZ) \cong \Bohr(\Ga)_0 \oplus \Prof(\Ga).$$
and that 
$$\Prof(\ZZ)= \varprojlim_{m} C_m \quad \text{and} \quad \Bohr(\ZZ)_0\cong \prod_{ \omega \in \frak{c}} \AA/\QQ,$$
where  $\AA/\QQ$ is the ring of adeles of $\QQ$ and $\frak{c}=2^{\aleph_0}$
 (see \cite[Proposition 11]{Bekka1}).
 
For an integer $n_0\geq 2,$ let  $G= C_{n_0} \wr \ZZ$ be the lamplighter group.
We claim that
$$
\Bohr(G)\cong \Bohr(\ZZ)_0 \times \Prof(G)
$$
and 
$$
\Prof(G)= \varprojlim_{m}  C_{n_0} \wr  C_m.
$$
Indeed, let  $N:=\oplus_{k\in \ZZ}C_{n_0}.$
It will be convenient to describe $N$ as the set  of maps $f: \ZZ\to C_{n_0}$
such that $\supp (f):=\{k\in \ZZ\ : f(k)\neq 0\}$ is at most finite.
The action of $m\in \ZZ$ on $f\in N$ is given 
by translation: $f^m(k)= f(k+m)$ for all $k\in \ZZ.$

We identify $\widehat{\Lambda}$ with  the group $\mu_{n_0}$
of $n_0$-th roots of unity in $\CCC$ by means of the duality 
$$\langle z, k\ZZ \rangle=  z^k \tout z\in \mu_{n_0}, k\in \ZZ.
$$
Then $\widehat{N}$ can be identified with 
the set of maps $\Phi: \ZZ\to \mu_{n_0},$ with duality given by
$$\langle \Phi, f \rangle=  \prod_{k\in \ZZ}\langle \Phi(k), f(k) \rangle \tout \Phi\in \widehat{N},  f\in N.$$
Observe that  $\Phi(N) \subset \mu_{n_0}$ and so  $\widehat{N}=\widehat{N}_{\rm finite}.$

We have $\widehat{N}^{H-{\rm per}}=\bigcup_{m\geq 1}\widehat{N}(m),$
where $\widehat{N}(m)$ is the subgroup
$$
 \widehat{N}(m)=\left \{\Phi: \ZZ\to \mu_{n_0}\ : \Phi (k+m)= \Phi(k) \quad \text{for all} \quad  k\in \ZZ\right\}.
$$
Observe that we have natural injections $i_{m_2}^{m_1}:  \widehat{N}(m_2) \to  \widehat{N}(m_1)$
if $m_1$ is a multiple of $m_2.$
The dual group $A(m)$ of $\widehat{N}(m)$ can be identified with 
the set of maps $\overline{f}:C_m\to C_{n_0}$ by means of the duality
 $$
 \langle \overline{f}, \Phi \rangle=  \prod_{k+m\ZZ\in C_m} \Phi(k)^{\overline{f}(k+m\ZZ)} \rangle \tout \Phi\in \widehat{N}(m),
 \overline{f}\in A(m).
$$
If $m_1$ is a multiple of $m_2,$ we have a projection 
$p_{m_1}^{m_2}:  A(m_1) \to A(m_1)$ given by 
$$\langle p_{m_1}^{m_2} (\overline{f}), \Phi\rangle=
 \langle \overline{f}, \Phi\circ i_{m_2}^{m_1} \rangle.$$
The dual group $A$ of $\widehat{N}^{H-{\rm per}}=\bigcup_{m\geq 1}\widehat{N}(m)$
can  then be identified with the projective limit $\varprojlim_{m} A(m)$.

The action of $ \ZZ$ by automorphisms of  $A$ is given, for $r\in \ZZ$ and
$\overline{f}=(\overline{f}_m)_{m\geq 1}\in A$ by $(\overline{f})^r= (\overline{g}_m)_{m\geq 1},$
where 
$$ \overline{g}_m(k+m\ZZ)= \overline{f}_m(k+r+m\ZZ) \tout k\in \ZZ.$$
This action extends to an action of ${\Proj}(\ZZ) = {\varprojlim}_m C_m$
by automorphisms on $A$ in an obvious way. 
By Corollary ~\ref{Cor1},  the group
$\Prof(G)$ is isomorphic to the corresponding semi-direct product $A \rtimes \Prof (\ZZ)$ and 
hence
$$
\Prof(G) \cong  \varprojlim_m  C_{n_0} \wr  C_m.
$$
By Corollary ~\ref{Cor1} again, the action of $\ZZ$ on $A$ extends to an 
action by automorphisms of $\Bohr(\ZZ).$
Since $\Bohr(\ZZ)_0$ is connected and $A$ is totally disconnected,
$\Bohr(\ZZ)_0$  acts as the identity on $A.$
Since $\Bohr(\ZZ)\cong \Bohr(\ZZ)_0 \times \Proj(\ZZ),$
it  follows that 
$$\Bohr(G) \cong  (A\rtimes \Proj(\ZZ)) \times \Bohr(\ZZ)_0 \cong {\Proj}(G) \times \Bohr(\ZZ)_0.$$

For another description of    $\Prof(G),$ see \cite[Lemma 3.24]{GrigorchukEtAl}.
\subsection{Heisenberg group}
 \label{SS:Heisenberg}
 Let $R$ be a commutative unital ring. 
The Heisenberg group   is the group 
 \[
H(R):=\left\{\left(
\begin{array}{ccc}
 1&a&c\\
 0&1&b\\
0&0&1\
\end{array}
\right): a,b, c\in R\right\}
\]
We can and will  identify $H(R)$ with $R^3$, 
equipped with the group law 
$$
(a,b,c) (a', b', c') \, = \, (a + a', b + b', c + c' + ab').
$$
Let $ {\mathcal I}_{\rm finite}$ be the family of \emph{ideals}  of $R$ with \emph{finite} index.
Every ideal $I$ from  ${\mathcal I}_{\rm finite}$ defines two compact groups $H(\Bohr(R), I)$
and $H(\Prof(R), I)$  of  Heisenberg type as follows: 
$$H(\Bohr(R),I):= \Bohr(R)\times \Bohr(R)\times (R/I)$$ is equipped with the group law 
$$
(x,y,z) (x', y', z') \, = \, (x + x', y + y', z + z' + p_I(x)p_I(y'),
$$
where $p_I: \Bohr(R)\to R/I$ is the group homomorphism induced by the canonical
map $R\to R/I;$ the group $H(\Prof(R),I)$ is defined in a similar way.

Observe  that, for two ideals $I$ and $J$ in  ${\mathcal I}_{\rm finite}$, we have
 natural epimorphisms 
 $$H(\Bohr(R),J)\to H(\Bohr(R),I) \qquad{and} \qquad H(\Prof(R),J)\to H(\Prof(R),I).$$

We claim that  the canonical maps $H(R)\to H(\Bohr(R),I)$ 
and $H(R)\to H(\Prof(R),I)$ induce   isomorphisms
$$
\Bohr(H(R))\cong \varprojlim_{I}  H(\Bohr(R), I),
$$
and 
$$
\Prof(H(R))\cong \varprojlim_{I}  H(\Prof(R), I),
$$
where $I$ runs over $ {\mathcal I}_{\rm finite}.$

Indeed, $H(R)$ is a semi-direct product 
$N\rtimes  H$ for  
$$N = \{ (0,b,c) : b, c \in R \} \cong R^2$$
and
$$H = \{ (a,0,0) : a \in R \} \cong R.$$

Let $\chi\in \widehat{N}$. Then $\chi = \chi_{\beta, \psi}$
for a unique pair $(\beta, \psi) \in (\widehat R)^2$,
where $\chi_{\beta, \psi}$ is defined by 
$$
\chi_{\beta, \psi}(0,b,c) \, = \, \beta(b) \psi(c)
\hskip.5cm \text{for} \hskip.2cm
b, c \in R.
$$ 
For $h = (a,0,0) \in H$, we have
$$
\chi_{\beta, \psi}^{h}(0,b,c) \, = \, \beta(b) \psi(a^{-1}b) \psi(c)
\, = \, \chi_{\beta \psi^a, \psi} (0, b, c)
\hskip.5cm \text{for} \hskip.2cm
b, c \in R ,
$$
where $\psi^a \in \widehat R$ is defined by $\psi^a(b) = \psi(a^{-1}b)$ for $b \in R$.
It follows that the $H$-orbit of $\chi_{\beta, \psi}$ is
$$
\{\chi_{\beta \psi^a, \psi} : a \in R \},
$$
and that the stabilizer of $\chi_{\beta, \psi}$, which only depends on $\psi$, is 
$$
H_\psi \, = \, \{ (a, 0, 0) \mid a \in I_\psi \},
$$
where $I_\psi$ is the ideal of $R$ defined by 
$$
I_\psi \, = \, \{ a \in R \mid aR \subset \ker \psi \}.
$$
Let ${\widehat R}_{\rm per}$ be the subgroup of  all $\psi \in\widehat R$
which factorizes through a quotient $R/I$ for an ideal $I\in {\mathcal I}_{\rm finite}.$
It follows that 
$${\widehat N}^{H-{\rm per}}=
\{\chi_{\beta, \psi} :
 \beta \in \widehat R, \psi \in {\widehat R}_{\rm per} \} \cong  {\widehat R} \times  {\widehat R}_{\rm per}.
$$
 The dual group of  ${\widehat R}_{\rm per}$ can be identified with $ \varprojlim_{I} R/I$,
 where $I$ runs over ${\mathcal I}_{\rm finite}.$
 So, the dual group $A$ of ${\widehat N}^{H-{\rm per}}$ can be identified
 with  $ \varprojlim_{I}  \Bohr(R) \times (R/I).$
 
 The action of $\Bohr(H)\cong \Bohr(R)$  on every $\Bohr(R) \times (R/I)$ is given by 
 $$
 x \cdot (y,z)=(y, z + p_I(x)p_I(y')) \tout x,y \in \Bohr(R), z\in R/I,
 $$
 for the natural map $p_I: \Bohr(R)\to R/I$.
 This shows that 
 $$\Bohr(H(R))\cong \varprojlim_{I}  H(\Bohr(R), I).$$
 
 Similarly, the dual group $B$ of ${\widehat N_{\rm finite}}^{H-{\rm per}}$ can be identified
 with  $ \varprojlim_{I}  \Prof(R) \times (R/I)$ and we have
 $$\Prof(H(R))\cong \varprojlim_{I}  H(\Prof(R), I).$$

\end{document}